\documentclass[11pt]{amsart}
\usepackage{setspace}
\usepackage{amsthm}

\usepackage[all]{xy}
\usepackage{amssymb}
\usepackage{enumerate}
\usepackage{mathrsfs}
\usepackage{epsfig}
\usepackage{graphicx}
\usepackage{caption}
\usepackage{subcaption}
\usepackage{float}
\usepackage{epigraph}

\numberwithin{equation}{section}

\newtheorem{thm}{Theorem}[section]

\newtheorem{lem}[thm]{Lemma}

\newtheorem{rem}{Remark}[section]
\newtheorem{example}[thm]{Example}

\newcommand{\eq}[1]{(\ref{#1})}

\renewcommand{\Re}{\operatorname{\rm Re}}
\renewcommand{\Im}{\operatorname{\rm Im}}

\newcommand{\beqast}{\begin{eqnarray*}}
\newcommand{\eqast}{\end{eqnarray*}}
\newcommand{\beqa}{\begin{eqnarray}}
\newcommand{\eqa}{\end{eqnarray}}

\newcommand{\bbe}{\begin{equation}}
\newcommand{\ee}{\end{equation}}

\renewcommand{\Re}{\operatorname{\rm Re}}
\renewcommand{\Im}{\operatorname{\rm Im}}

\newcommand{\bR}{{\mathbb R}}
\newcommand{\bT}{{\mathbb T}}

\newcommand{\bZ}{{\mathbb Z}}

\newcommand{\cD}{{\mathcal D}}

\newcommand{\cL}{{\mathcal L}}

\newcommand{\cC}{{\mathcal C}}

\newcommand{\cU}{{\mathcal U}}

\newcommand{\cZ}{{\mathcal Z}}

\newcommand{\tV}{{\tilde V}}

\newcommand{\tu}{{\tilde u}}
\newcommand{\tv}{{\tilde v}}

\newcommand{\al}{\alpha}

\newcommand{\be}{\beta}

\newcommand{\de}{\delta}
\newcommand{\eps}{\epsilon}
\newcommand{\ka}{\kappa}
\newcommand{\la}{\lambda}

\newcommand{\La}{\Lambda}

\newcommand{\sg}{\sigma}

\newcommand{\om}{\omega}

\newcommand{\ze}{\zeta}

\newcommand{\ga}{\gamma}
\newcommand{\gap}{\gamma_+}
\newcommand{\gam}{\gamma_-}

\newcommand{\Ga}{\Gamma}

\begin{document}

\title[Efficient inverse $Z$-transform and Wiener-Hopf factorization]
{Efficient inverse $Z$-transform and Wiener-Hopf factorization}
\author[
Svetlana Boyarchenko and
Sergei Levendorski\u{i}]
{
Svetlana Boyarchenko and
Sergei Levendorski\u{i}}

\begin{abstract}
We suggest new closely related methods for numerical inversion of $Z$-transform
and Wiener-Hopf factorization of functions on the unit circle, based on sinh-deformations of the contours of
integration, corresponding changes of variables and  the simplified trapezoid rule.
As applications, we consider evaluation of high moments of probability distributions and 
 construction of causal filters.
Programs in Matlab running on a Mac with moderate characteristics achieves the precision E-14 in several dozen of microseconds and E-11 in several milliseconds, respectively. 
\end{abstract}

\thanks{
\emph{S.B.:} Department of Economics, The
University of Texas at Austin, 2225 Speedway Stop C3100, Austin,
TX 78712--0301, {\tt sboyarch@utexas.edu} \\
\emph{S.L.:}
Calico Science Consulting. Austin, TX.
 Email address: {\tt
levendorskii@gmail.com}}

\maketitle



%\tableofcontents

\section{Introduction} 
The Fourier and Laplace transforms, their discrete analogs and Wiener-Hopf factorization 
 are widely used
in various fields of mathematics, natural sciences, engineering, statistics, finance and economics. 
In many important cases of interest, numerical realizations of the integrals in the formulas for the Fourier/Laplace inversion 
and Wiener-Hopf factors are far from trivial due to high oscillation and/or slow decay at infinity of the integrands.
Additional difficulties arise if the integrands are not smooth as in the formulas for probability distributions
in stable L\'evy models and transfer functions in causal filtering of highly persistent shocks.
In the result, many popular methods, FFT in particular, produce serious errors and/or are extremely time consuming.
 Examples of typical errors in the context of evaluation of probability distributions
and option pricing can be found in \cite{iFT,pitfalls,SINHregular}.
Conformal deformations of the contours of integration alleviate this problem. See, e.g.,
\cite{Talbot79,Fedoryuk,stenger-book,AbWh92OR,AbWh06,AbateValko04,TrefethenWeidmanSchmelzer,TrefethenWeidman07,TrefethenWeidmanTrapezoid14} and the bibliographies therein. 
After an appropriate contour deformation, very efficient quadratures such as  Gauss
and Glenshaw-Curtis quadratures are applied.   
The weights and nodes must be accurately
precalculated, and the weights can be rather large. To avoid large rounding errors,  the integrands must be evaluated very accurately. In a number of situations of interest, the integrands are expressed in terms of special functions or evaluated using the Wiener-Hopf technique, hence, sufficiently accurate evaluation is difficult and time
consuming; high precision arithmetic might be needed. Finally,  
 generalizations of the constructions to new integrals and multi-dimensional integrals are far from straightforward.
 
 In the present paper, we apply the sinh-acceleration and other conformal accelerations to the inverse $Z$-transform and Wiener-Hopf factorization of functions
on the unit circle $\bT$. In different fields of science,  two versions of the $Z$-transform of a series $\{u_n\}_{n\in \bZ_+}$ are used:
$\tu(z)=\sum_{n\in \bZ_+}u_n z^{-n}$ and $\tu(z)=\sum_{n\in \bZ_+}u_n z^{n}$.
We use the latter version in Sect. \ref{s:invZ}-\ref{s:sinh_II}, and the former in applications to signal processing in Section
\ref{s:ZWHFfilter}. Replacing $\bZ_+$ with $\bZ$, one obtains the two-sided $Z$ transform.
 In  examples in Sect.  \ref{s:invZ}-\ref{s:sinh_II}, we   calculate  moments of probability distributions, which is a standard tool for identification of
 the latter (see, e.g, \cite{Feller_book, StoyanovCounterExample,ChoudLucanton,KyrBrignFusai23}). The key properties 
 that we use are
 the decay of the moment-generating function in a cone around $(-\infty,0]$ or around the imaginary axis; the efficiency of the numerical schemes increases with the opening angle of the cone. As it is demonstrated in \cite{SINHregular,EfficientAmenable}, wide classes of popular distributions enjoy  one of these properties. In examples, we
 use the distributions of KoBoL processes constructed in \cite{KoBoL}.
 If the distribution has atoms, then the sinh-acceleration is not  applicable but a different family of conformal maps can be used. In Sect.~\ref{s:ZWHFfilter}, we consider the Wiener-Hopf factorization of functions on a circle in applications to causal filtering
 (see, e.g., \cite{BrownHwang96,GasWit98,StoicaMosesSpectrAnalSignals05}),
 and explain how the sinh-deformation of the contour of integration in the formulas for the Wiener-Hopf factors
decreases the complexity of the numerical scheme and facilitates the application of efficient inverse $Z$-transform
to the calculation of transfer functions.

The paper is a natural extension of a series of papers
  \cite{SINHregular,Contrarian,EfficientStableLevyExtremum,EfficientDoubleBarrier2,Joint-3}, where we used  simple families of sinh-deformations and the corresponding conformal change of variables, in two versions: $\xi=i\om_1+b\sinh(i\om+y)$ and $z=\sg_\ell+ib_\ell\sinh(i\om_\ell+y)$ for the Fourier and Laplace inversion, respectively. In the new variables, the integrands
are analytic in  strips around the real axis, hence,  
the  simplified trapezoid rule is efficient. In the one-dimensional (1D) case, the rate of convergence of the resulting numerical scheme 
is  worse than the rates of convergence in \cite{SchmelzerTrefethen07,TrefethenWeidmanSchmelzer,TrefethenWeidman07,TrefethenWeidmanTrapezoid14}.
 However, the sinh-deformation technique is easier to apply to complicated integrals arising in probability and finance,
 and no precalculation of the nodes and (large) weights is necessary. Typically, it is possible to choose a deformation
 such that the integrand is not very large, and decays  exponentially or faster at infinity. Hence, the discretization and truncation errors are fairly easy to control. 
The general scheme for the choice of appropriate sinh-deformations, number of steps and step sizes
has been successfully applied to the Fourier-Laplace inversions in dimensions 1-5, which is a rather challenging
task for the methods that are most efficient in 1D; the integrands are expressed in terms
of the Wiener-Hopf factors, and the latter are calculated using the same family of deformations. 
With respect to some variables, the Gaver-Whynn-Rho algorithm for the Laplace inversion is used.
The modification of the same scheme is successfully applied to the evaluation of probability distributions, special functions and pricing exotic options in stable L\'evy models in \cite{ConfAccelerationStable,EfficientLevyExtremum}. Instead of  
the sinh-deformations, appropriate rotations of the axis of integration and exponential changes of variables
are used; in some cases such as non-symmetric stable L\'evy models of index $\al=1$ or close to 1,
additional families of conformal deformations are used.

We remind to the reader basic formulas, error bounds and recommendations for the choice of the number of terms in the trapezoid rule in Sect. \ref{s:invZ}
(more involved and detailed error bounds and recommendations can be found in \cite{ChoudLucanton}). In Sect.~
\ref{s:sinh_I}-\ref{s:sinh_II}, 
 we construct several versions of the numerical scheme for the inverse $Z$-transform, introduced in \cite{EfficientDiscExtremum}.
  We produce examples to demonstrate why several versions are needed, and outline additional useful modifications in Sect. \ref{s:concl}.  The Wiener-Hopf factorization with applications to filtering of highly persistent shocks are in
  Sect.~\ref{s:ZWHFfilter}. Sect. \ref{s:concl} concludes.

\section{Inverse $Z$-tranform and trapezoid rule}\label{s:invZ}
Let $u\in l_1(\bZ)$. Then $\tu$ is continuous function on $\bT$, and
the inverse $Z$ transform is given by 
\bbe\label{eq:invZ}
u_n=\frac{1}{2\pi i}\int_{\bT}\tu(z)z^{-n-1}dz.
\ee
Fix $n$, define
$h(z)=\tu(z)z^{-n}$, denote by $I(h)$ the RHS of \eq{eq:invZ}, and approximate
  $I(h)$ by 
\bbe\label{defTM}
T_N(h) = (1/N) \sum_{k=0}^{N-1} h(\zeta_N^k),
\ee
where $N>1$ is an integer, and $\zeta_N = \exp(2\pi i/N)$ is the standard primitive $N$-th root of unity.
 If $h$ admits analytic continuation to an annulus $\cD{(a_-,a_+)}:=\{z\ |\ a_-<|z|<a_+\}$, where $0\le a_-<a_+\le +\infty$, the trapezoid rule converges exponentially. See \cite{ChoudLucanton} for various versions of the error bounds;
 we use the simplest one:
 %and the following error bound is well-known
 %; for completeness, we give the proof in Sect. \ref{ss:proof_disctraperror}.
\begin{thm}\label{disctraperror}
Let $h$ be analytic in  $\cD_{(1/\rho,\rho)}$, where $\rho>1$, and let 
the Hardy norm of $h$
\[
\|h\|_{\cD_{(1/\rho,\rho)}}=\frac{1}{2\pi i}\int_{|z|=1/\rho}|h(z)|\frac{dz}{z}+\frac{1}{2\pi i}\int_{|z|=\rho}|h(z)|\frac{dz}{z}
\]
be finite. 
Then 
\bbe\label{errtrapgen}
|T_N(h)-I(h)|\le \frac{\rho^{-N}}{1-\rho^{-N}}\|h\|_{\cD_{(1/\rho,\rho)}}.
\ee
 \end{thm}
%If $u_n$ are real, then $\overline{h(z)}=h(\bar z)$, hence, we can choose an odd  $N=2N_0+1$ 
%and obtain
%replace \eq{defTM} with
%\bbe\label{defTMsym}
%T_N(h) = (2/N) \Re\sum_{k=0}^{N-1} h(\zeta_N^k)(1-\de_{k0}/2).
%\ee 
To satisfy a small error tolerance $\eps>0$, it is necessary to choose $N=N(\eps,n)$ and $\rho>1$ so that $\rho^{-N}$ is small, hence,
we may use an approximate bound
\bbe\label{errtrapgen_app}
|T_N(h)-I(h)|\le \rho^{-N}\|h\|_{\cD_{(1/\rho,\rho)}}.
\ee 
%We make several remarks. (1) 
\begin{rem}\label{rem:trap}
{\rm 
\begin{enumerate}[(1)]
\item
If either $1\ge a_-<a_+$ or $0\le a_-<a_+\le 1$, then the rescaling $z=rz'$ with an appropriate $r>1$ and $r\in (0,1)$, respectively,
can be used to reduce to the case $0\le a_-<1<a_+$. % (2) 
\item
If either $a_-=0$ or close to 0 or $a_+=\infty$ or very large, then the rescaling can be used to reduce to the case when a large $\rho$ can be chosen. However, then, in the case of large $n$, the Hardy norm is very large. Thus, one is forced to use $\rho$ close to 1.
% to avoid large $N$ and large rounding errors.
%(3)
\item
 If analytic continuation to an annulus containing $\bT$ is impossible,
 then only  real-analytical error bounds are applicable,
and the rate of convergence of the trapezoid rule is very poor.
%(4) 
\item
If $a_+/a_--1$ is small, then, after an approximately optimal rescaling, we can reduce to the case $(a_-,a_+)=(1/\rho,\rho)$,
where $\rho=\sqrt{a_+/a_-}$ is close to 1: $\rho=e^\de$, where $0<\de<<1$. Assume for simplicity that 
$\|\tu\|_{\cD_{(1/\rho,\rho)}}<\infty$ (if not, one chooses a smaller $\rho$). Then, as $N\to\infty$,
 the RHS of the error bound \eq{errtrapgen_app} is of the order of $e^{(n-N)\de}$. Hence, to satisfy a small error tolerance
 $\eps$, one is forced to use a large $N\approx n+(1/\de)E$, where $E=\ln(1/\eps)$.
 If $N$ is large, then, to avoid rounding errors, it can be necessary to evaluate $\tu(z)$ with high precision. This is especially time consuming
 if $\tu$ is given by complicated expressions in terms of special functions as in \cite{KyrBrignFusai23} where the moments of probability distributions are calculated or evaluated using the Wiener-Hopf factorization
 technique as in \cite{FusaiBarr,FusaiGermanoMarazzina,CernyKyriakou, EfficientDiscExtremum} where exotic options are priced. In the latter case, for each value of $z$ used in the trapezoid rule, $\tu(z)$ is evaluated in terms
 of double integrals. The integrands are expressed in terms of the Wiener-Hopf factors, which are 
expressed in terms of certain integrals. The evaluation of the latter is time-consuming.
\end{enumerate} }\end{rem}
In the case of the (one-sided) $Z$-transform, only  $n\ge 0$ need to be considered. In the 
general case, we assume that $n\ge 0$. The case $n<0$ reduces to the case $n>0$ by changing the variable
$z\mapsto 1/z$. 
Assume that 
$\tu$ is analytic in the interior of $\bT$, and the interior is the maximal disc of analyticity. 
Hence, we use \eq{eq:invZ} with integration over $\{z\ |\ |z|=r\}$, where $r\in (0,1)$. Changing the variable
$z\mapsto rz$ and letting $\tu_r(z)=\tu(rz)r^{-n-1}$, we obtain
\bbe\label{eq:invZr}
u_n=\frac{1}{2\pi i}\int_{\bT}\tu_r(z)z^{-n-1}dz.
\ee
We take  $\rho\in (1,1/r)$, and apply the bound \eq{errtrapgen} with $h_r(z)=\tu_r(z)z^{-n}$. The Hardy norm is
\[
\|h\|_{\cD(r_-,r_+)}=\frac{1}{2\pi i}\int_{|z|=1/\rho}|\tu_r(z)z^{-n}|\frac{dz}{z}+\frac{1}{2\pi i}\int_{|z|=\rho}|\tu_r(z)z^{-n}|\frac{dz}{z},
\]
where $r_-:=r/\rho<r<r_+:=r\rho<1$. The choice of $r\in (0,1)$ being arbitrary, we can use arbitrarily large $\rho$.
However, if $\rho$ is large, one has to multiply very small numbers by very large ones, and high precision arithmetic is necessary to avoid large rounding errors. 

We make a realistic assumption that one can  evaluate the terms in
  the trapezoid rule sufficiently accurately only if the terms are not too large. We impose the condition
  on the admissible size of the terms in the form $r^{-n}\le e^M$. The following  approximation $N_{appr}=N_{appr}(\eps,n,M)$ to $N=N(\eps,n,M)$ in terms of $E=\ln(1/\eps), n$ and $M$ \footnote{We write $N(\eps,n,M)\approx N_{appr}(\eps,n,M)$ if there exist $c,C>0$ independent of $(\eps,n,M)$
  such that $cN_{appr}(\eps,n,M)\le N(\eps,n,M)\le CN_{appr}(\eps,n,M)$.} is derived in \cite{EfficientZsufficient}.
    \begin{lem}\label{lem:lower_bound_comp_trap}
  Let there exist $C_0>c_0>0$ such that 
  \bbe\label{eq:bound_main_trap}
  c_0|1-z|^{-1}\le |\tu(z)|\le C_0|1-z|^{-1}, \ z\in \cD(0,1),
  \ee
  and let $r^{-n}= e^M$, where $M$ is independent of $n$. 
  
  Then, if $n>>1$, $\eps<<1,$  $n>> E>> \ln n$ and $E>>\ln M$,
  \bbe\label{Neps_n_M} 
  N(=N(\eps,n,M))\approx N_{appr}:= \frac{n}{M}(E+2M).%n(E+\ln M+\ln\ln n)/M.
  \ee
  \end{lem}
 %It is seen from \eq{Neps_n_M} that  $N$ decreases as $M$ increases. 
 For moderately large $M'$s and very large $n$, $N$ is very large.
  
  \section{Sinh-acceleration I}\label{s:sinh_I} \subsection{General formulas and bounds}\label{ss:sinhIgen}
We
deform the contour of integration $\{z=e^{i\varphi}\ |\ -\pi<\varphi<\pi\}$ into a contour of the form
$\cL_{L; \sg_\ell,b_\ell,\om_\ell}=\chi_{L; \sg_\ell,b_\ell,\om_\ell}(\bR)$: 
 \bbe\label{eq:invZsinh}
u_n=\frac{1}{2\pi i}\int_{\cL_{L; \sg_\ell,b_\ell,\om_\ell}}\tu(z)z^{-n-1}dz,
\ee
where $b_\ell>0$, $\sg_\ell\in\bR$, $\om_\ell\in (-\pi/2,\pi/2)$, and the map $\chi_{L; \sg_\ell,b_\ell,\om_\ell}$
is defined by \bbe\label{eq:sinhLapl}
\chi_{L; \sg_\ell,b_\ell,\om_\ell}(y)=\sg_\ell +i b_\ell\sinh(i\om_\ell+y).
\ee 
  We make the change of variables
 \bbe\label{izeT0sinh}
u_n=\int_{\bR}\frac{b_\ell}{2\pi}\chi_{L; \sg_\ell,b_\ell,\om_\ell}(y)^{-n-1}\cosh(i\om_\ell+y)\tu(\chi_{L; \sg_\ell,b_\ell,\om_\ell}(y))dy,
\ee
denote by $f_n(y)$ the integrand on the RHS of \eq{izeT0sinh}, apply the infinite trapezoid rule
\bbe\label{Vn_inf_sinh}
u_n\approx \ze_\ell \sum_{j\in \bZ}f_n(j\ze_\ell),
\ee
and truncate the series replacing $\sum_{j\in \bZ}$ with $\sum_{|j|\le N}$. 
If $u_n$ are real, then the number of terms of the simplified trapezoid rule can be decreased almost two-fold.
%\bbe\label{Vn_inf_sinh_sym}
%u_n\approx 2\ze_\ell \Re\sum_{0\le j\le N}f_n(j\ze_\ell)(1-\de_{j0}/2).
%\]
The deformation is possible and a simple error bound is available under certain conditions on
 the domain of analyticity and  rate of decay of $\tu(z)$ as $z\to \infty$.  For $\al\in(0,\pi)$, denote 
 $\cC_{\al}=\{\rho e^{i\varphi}\ |\ |\varphi|<\al, \rho>0\}$.
 
\vskip0.1cm
\noindent
  {\sc Condition $Z$-SINH1$(a_-,a_+;\al)$.}  
 There exist $0\le a_-<1\le a_+$ and  $\al\in (0,\pi)$  such that
 \begin{enumerate}[(a)] \item 
 $\tu$ admits analytic continuation
 to
  \[
  \cU(a_-,a_+,\al):=\cD(a_-,a_+)\cup((a_+-\cC_{\al})\setminus
 \{z |\ |z|\le a_-\});
 \]
 \item
 for any $a_-<r_-<r_+<a_+$ and $\al_1\in (0,\al)$,
  \bbe\label{eq:main_bound}
|\tu(z)|\le C_\tu(r_-,r_+;\al_1)(1+|z|)^{m_\tu}, \
z\in \cU(r_-, r_+;\al_1),
\ee
where $m_\tu$ depends only on $\tu$, and $C_\tu(r_-,r_+;\al_1)$ depends on $(r_-,r_+,\al_1)$.
\end{enumerate} 
\begin{thm}\label{thm:sinhI} Let Condition $Z$-SINH1$(a_-,a_+;\al)$ hold and $n>m_\tu$. Then
\begin{enumerate}[(a)]
\item
for any $\om_\ell\in (\pi/2-\al,\pi/2)$, there exist $\sg_\ell>0$ and $b_\ell>0$   such that 
$b_\ell>\sg_\ell\sin(\om_\ell)$, and
$\sg_\ell-b_\ell\sin(\om_\ell)\in (a_-, a_+)$;
\item
if   a triplet $(\sg_\ell,b_\ell,\om_\ell)$ satisfies the conditions in (a),
then the distance from $\cL_{L;\sg_\ell,b_\ell,\om_\ell}$ to the origin equals $\sg_\ell-b_\ell\sin(\om_\ell)$,
 $\cL_{L; \sg_\ell,b_\ell,\om_\ell}\subset\cU(a_-,a_+,\al)$ and \eq{eq:invZsinh} is valid; 

\item
if $\al\in (\pi/2,\pi)$, then, for any $\om_\ell
\in (\pi/2-\al, 0]$, $b_\ell>0$ and $\sg_\ell\in (a_-+b_\ell\sin(\om_\ell), a_++b_\ell\sin(\om_\ell))$, $\cL_{L; \sg_\ell,b_\ell,\om_\ell}\subset\cU(a_-,a_+,\al)$ and \eq{eq:invZsinh} is valid;
\item
in both cases (b) and (c), there exists $d_\ell>0$ such that the image of the strip $S_{(-d_\ell,d_\ell)}$ under $\chi_{L; \sg_\ell,b_\ell,\om_\ell}$ is a subset of $\cU(a_-,a_+,\al)$, and the distance from $\chi_{L;\sg_\ell,b_\ell,\om_\ell}(S_{(-d_\ell,d_\ell)})$ to the origin equals $\sg_\ell-b_\ell\sin(\om_\ell+d_\ell)$.

\end{enumerate}
\end{thm}
\begin{proof} (a) If $\om_\ell\le 0$, we take any $b_\ell\in (0,a_-)$ and $\sg_\ell\in (a_-+b_\ell\sin(\om_\ell),
a_++b_\ell\sin(\om_\ell))$. If $\om_\ell>0$, we take any $\sg_\ell\in (a_-\cos^{-2}(\om_\ell), a_+\cos^{-2}(\om_\ell))$.
For $b^0_\ell=\sg_\ell\sin(\om_\ell)$, $\sg_\ell-b^0_\ell\sin(\om_\ell)=\sg_\ell\cos^2(\om_\ell)\in (a_-,a_+)$. Hence,
$\sg_\ell-b_\ell\sin(\om_\ell)\in (a_-,a_+)$ if $b_\ell-b^0_\ell>0$ is sufficiently small.

 (b,c)  If $\om_\ell>0$ and $\sg_\ell<b_\ell/\sin\om_\ell$, 
then it is straightforward to prove that
\bbe\label{eq:cLp}
\mathrm{dist}\, (0, \cL_{L; \sg_\ell,b_\ell,\om_\ell})=\sg_\ell-b_\ell\sin(\om_\ell),
\ee
therefore, $\cL_{L; \sg_\ell,b_\ell,\om_\ell}\cup \cD(0,a_-)=\emptyset$, and $\cD(0,a_-)$ is to the left of $\cL_{L; \sg_\ell,b_\ell,\om_\ell}$. If $\om_\ell\le 0$,  the region to the right
of $\cL_{L; \sg_\ell,b_\ell,\om_\ell}$ is convex. We note that $\cL_{L; \sg_\ell,b_\ell,\om_\ell}$ intersects the real axis at $\sg_\ell-b_\ell\sin(\om_\ell)\in (a_-,a_+)$, and we conclude that $\cD(0,a_-)$ is to the left of $\cL_{L; \sg_\ell,b_\ell,\om_\ell}$.

Next,
 $\sg_\ell-b_\ell\sin(\om_\ell)+e^{ i(\pi/2+\om_\ell)}\bR_+$ is the asymptote of $\cL_{L; \sg_\ell,b_\ell,\om_\ell}$ in the upper half-plane. The asymptote is below the ray $a_++e^{i(\pi-\al)}\bR_+$ since $\sg_\ell-b_\ell\sin(\om_\ell)<a_+$ and
 $\pi/2+\om_\ell>\pi-\al$, and the region to the left of 
 $\cL_{L; \sg_\ell,b_\ell,\om_\ell}$ is convex. Hence,  $\cL_{L; \sg_\ell,b_\ell,\om_\ell}$
 is to the left of the right boundary of $\cU(a_-,a_+,\al)$. Thus, $\cL_{L; \sg_\ell,b_\ell,\om_\ell}\subset \cU(a_-,a_+,\al)$,
 and since $m_\tu-n-1<-1$, we can deform
the contour $\{z=e^{i\varphi}\ |\ -\pi<\varphi<\pi\}$, first, into 
$((-\infty,-1]-i0]\cup \{z=e^{i\varphi}\ |\ -\pi<\varphi<\pi\}\cup (((-\infty,-1]+i0]$, and then  into $\cU(a_-,a_+,\al)$ remaining in $\cU(a_-,a_+,\al)$.

(d) If $d_\ell$ is sufficiently small in absolute value then  
$\cL_{L; \sg_\ell,b_\ell,\om_\ell\pm d_\ell}\subset \cU(a_-,a_+,\al)$. The proof is the same as for $d_\ell=0$.
We take into account that if $\sg_\ell<b_\ell/(2\sin(\om_\ell))$ and $d_\ell>0$ is sufficiently small, then $\sg_\ell<b_\ell/(2\sin(\om_\ell+d_\ell))$ as well.
\end{proof}

Let $\sg_\ell, b_\ell, \om_\ell, d_\ell>0$ and $n$ satisfy the conditions of Theorem \ref{thm:sinhI}. Then $f_n$ admits analytic continuation to the strip $S_{(-d_\ell,d_\ell)}$ and exponentially decays as $y\to \infty$ remaining in the strip.
It follows that $f_n$ satisfies the conditions of the following lemma.

\begin{lem}[\cite{stenger-book}, Thm.3.2.1] 
Let $f_n$ be analytic in the strip $S_{(-d,d)}$, \\ $\lim_{R\to \pm\infty}\int_{-d}^d |f_n(i s+R)|ds=0,$
and 
\bbe\label{Hnorm}
H(f_n,d):=\|f_n\|_{H^1(S_{(-d,d)})}:=\lim_{s\downarrow -d}\int_\bR|g(i s+ t)|dt+\lim_{s\uparrow d}\int_\bR|g(i s+t)|dt<\infty.
\ee Then
the error of the infinite trapezoid rule admits an upper bound 
\bbe\label{Err_inf_trap}
\left|u_n- \ze_\ell \sum_{j\in \bZ}f_n(j\ze_\ell)\right|\le H(f_n,d)\frac{\exp[-2\pi d/\ze]}{1-\exp[-2\pi d/\ze]}.
\ee
\end{lem}
Note that the norm $H(f_n,d)$ is similar but not identical to the Hardy norm of $f_n$ defined on $S_{(-d,d)}$.
We will refer to $H(f_n,d)$ as a quasi-Hardy norm (q-Hardy norm).

 Once
an approximate bound $H_{\mathrm{appr.}}(f_n,d_\ell)$ for $H(f_n,d_\ell)$ is derived, it becomes possible to satisfy the desired error tolerance $\eps>0$
with a good accuracy letting
\bbe\label{rec_ze_ze}
\ze_\ell=\frac{2\pi d_\ell}{\ln(H_{\mathrm{appr.}}(f_n,d_\ell)/\eps)}.
\ee
Since $f_n(y)$ decays as $((b_\ell/2)e^{|y|})^{m_\tu-n}$ as $y\to\pm \infty$, it is straightforward to choose the truncation of the infinite sum on the RHS of \eq{Vn_inf_sinh}:
\bbe\label{Vn_inf_sinh_trunc}
u_n\approx \ze_\ell \sum_{|j|\le N_\ell}f_n(j\ze_j)
\ee
 to satisfy the given error tolerance. A good approximation to  $\La:=N_\ell\ze_\ell$ is 
\bbe\label{eqLa_z}
\La=\La(n-m_\tu, C_\tu, b_\ell, \eps):=\frac{1}{n-m_{\tu}}\ln\frac{C_{\tu}}{\eps}-\ln (b_\ell/2)+\La_0,
\ee
where $C_{\tu}$ and $m_{\tu}$ are from \eq{eq:main_bound}, and $\La_0$ is the length of 
$\cL_{L; \sg_\ell,b_\ell,\om_\ell}\cup \cD(0,1)$.  Thus, \bbe\label{N0}
N_\ell\approx (2\pi d_\ell)^{-1}(\ln H_{\mathrm{appr.}}(f_n,d_\ell)+\ln(1/\eps))\La(n-m_\tu, C_\tu, b_\ell, \eps).
\ee

\subsection{Parameter choice}\label{param_choice_Z_SINH} 
The construction of admissible deformation is simple in the case $\al\in (\pi/2, \pi)$.
Fix $a_-<r_-<r_+<a_+$, and  take $\om_\ell\in (\pi/2-\al, \pi/4)$ and
$d_\ell>0$ such that $0\le \om_\ell-d_\ell<\om_\ell+d_\ell<\pi/2-\al$. Since the step of the 
infinite trapezoid rule increases with $d_\ell$ (provided the q-Hardy norm does not increase too fast), it is approximately  optimal to choose $\om_\ell$ so that
the ray $ie^{i\om_\ell}\bR_+$ bisects the angle $e^{i(\pi-\al)}\bR_+\cup e^{i3\pi/4}\bR_+$.  Thus, we set  
$\om_\ell=-\pi/2+0.5(3\pi/4+\pi-\al)=3\pi/8-\al/2$, define
$d^0_\ell:=3\pi/4-(\pi/2+\om_\ell)=\al/2-\pi/8$, and set $d_\ell=k_dd_\ell^0$, where $k_d<1$ is close to 1, e.g., $k_d=0.9$.
Then we find $\sg_\ell\in \bR$ and $b_\ell>0$ such that $\sg_\ell-b_\ell\sin(\om_\ell-d_\ell)=r_+$,  $\sg_\ell-b_\ell\sin(\om_\ell+d_\ell)=r_-$.
 The straightforward calculations give 
\beqa\label{bell_sinh_L}
b_\ell&=&\frac{r_+-r_-}{\sin(\om_\ell+d_\ell)-\sin(\om_\ell-d_\ell)}=\frac{r_+-r_-}{2\cos(\om_\ell)\sin(d_\ell)},
\\\label{sgell_sinh_L}\sg_\ell&=&\frac{r_+\sin(\om_\ell+d_\ell)-r_-\sin(\om_\ell-d_\ell)}{2\cos(\om_\ell)\sin(d_\ell)},
\eqa 
and $r:=\sg_\ell-b_\ell\sin(\om_\ell)\in (r_-,r_+)$. Note that we need to choose $a_-<r_-<r_+<a_+$ so that the distance from the left boundary
$\cL_{L;\om_\ell,b_\ell,\om_\ell+d_\ell}$ to the origin equals $\sg_\ell-b_\ell\sin(\om_\ell+d_\ell)$. As we have proved
in Theorem \ref{thm:sinhI}, an equivalent condition is $b_\ell> \sg_\ell\sin(\om_\ell+d_\ell)$. Hence, the recommendation for the choice of $\om_\ell, d_\ell$ needs an adjustment. If $\al>\pi/2$, we choose $\om_\ell=(\al-\pi/2)/2$,
$d_\ell =-k_d\om_\ell$. In Examples \ref{example:SINH1simple} and \ref{example:SINH1} below, we make this simplifying choice  to compare the complexities of the trapezoid rule and sinh-acceleration algorithm. 
If $\om_\ell+d_\ell\ge 0$, 
 we derive an equivalent condition on $r_\pm$, $\om_\ell$ and $d_\ell$:
 \bbe\label{cond_rpm}
 r_-(1-\sin(\om_\ell+d_\ell)\sin(\om_\ell-d_\ell))< r_+(1-\sin^2(\om_\ell+d_\ell))
 \ee
 using \eq{bell_sinh_L}-\eq{sgell_sinh_L}.
If $a_-=0$ or sufficiently small, we can choose $r_\pm$ so that \eq{cond_rpm} holds. However,
 then $r_-$ can be small, 
but if $n$ is large,  it is necessary to choose $r_-<r_+$ very close to 1. Then  \eq{cond_rpm} can be satisfied
only if $\om_\ell+d_\ell$ is very small. We use an approximation
\bbe\label{apprpm}
(\om_\ell+d_\ell)^2-r_-(\om_\ell^2-d_\ell^2)\approx \de/2,
\ee
where $\de=r_+-r_-$.  With a simple choice $d_\ell=2\om_\ell/3$, \eq{apprpm} gives
$\om_\ell=c\sqrt{\de}$, where $c=\sqrt{9/48}$. In Example \ref{ex:one-sidedKBL-moment} below, it is necessary to choose $\om_\ell$ and $d_\ell>0$ so that $\om_\ell-d_\ell>0$, hence, we use this prescription.

\subsection{Examples}\label{ss:examplesZ-SINH1}
\begin{example}\label{example:SINH1simple}{\rm 
In Fig.~\ref{fig:Z_SINH1.pdf} (A), we plot nodes used for the evaluation of the 100-th moment of the distribution of a KoBoL subordinator, with the moment generating function $M(z)=e^{\Psi(z)}$; $\Psi(z)=c\Ga(-\nu)((\la-z)^\nu-\la^\nu)$, $c=0.1, \nu=0.5, \la=1.01$. Moment $\mu_{100}=5.32400799771669E-05$; the difference between values obtained with the two algorithms is smaller that $E-15$.
 Solid line: 1101 nodes used in the trapezoid rule. The number is chosen by hand as an approximately minimal number
 which satisfies the error tolerance $\eps=10^{-15}$.
  Dots: 33 nodes used in the sinh-acceleration, with the parameters
 $\om_\ell=-0.7854 $, $d_\ell=0.7069$. Parameters
 $\sg_\ell=0.978291504$ and $b_\ell=0.021775623$ are calculated using \eq{sgell_sinh_L} and \eq{bell_sinh_L} 
 with $r_+=1,r_-=0.98$. Step $\ze_\ell=0.1187$ is calculated using \eq{rec_ze_ze} with  
 an approximation $H_{\mathrm{appr.}}(f_n,d_\ell)=r_-^{-n}+10$. 
The prescription \eq{eqLa_z} gives unnecessary large $\La$, which we decrease by the factor 0.75.
 CPU times: 146 and 12 microsec. for the trapezoid rule and sinh-acceleration, respectively, the average over 100,000 runs.
 $\mu_{500}= 8.87234294030321E-08$
 can be calculated with the accuracy better than E-15 using the trapezoid rule with 1801 nodes and sinh-acceleration with
 30 nodes; the CPU times are 177 and 12 microsec., respectively, the average over 100,000 runs.
 \footnote{The calculations in the paper
were performed in MATLAB R2023b-academic use, on a MacPro Chip Apple M1 Max Pro chip
with 10-core CPU, 24-core GPU, 16-core Neural Engine 32GB unified memory,
1TB SSD storage.}

   \begin{figure}
\begin{tabular}{cc}

 \begin{subfigure}[h]{0.45\textwidth}

 \centering
    \includegraphics[width=0.9\textwidth,height=0.4\textheight]{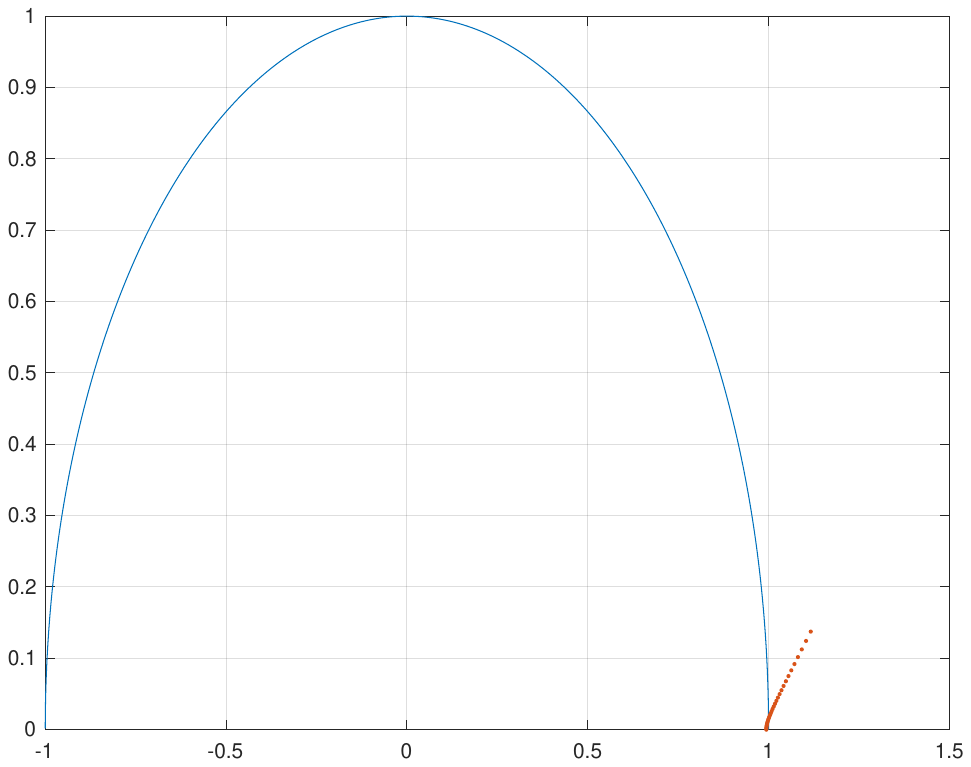}
    \caption{Illustration for Example \ref{example:SINH1simple}}\label{A}
\end{subfigure}
&
\begin{subfigure}[h]{0.45\textwidth}
\centering
    \includegraphics[width=0.9\textwidth,height=0.4\textheight]{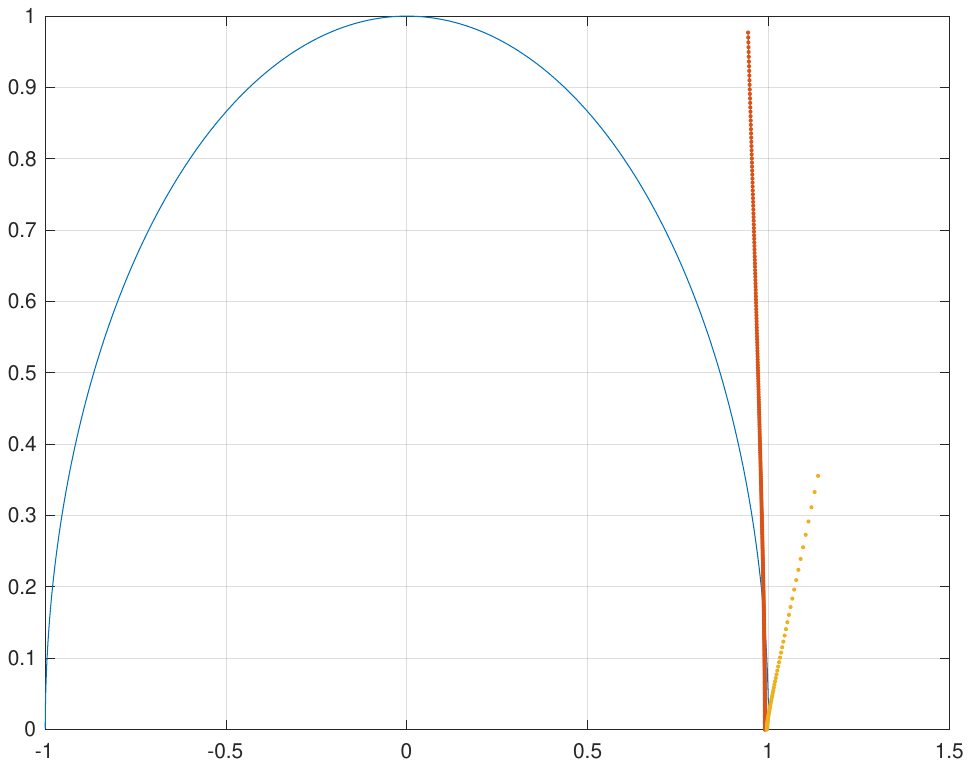}
    \caption{Illustration for Examples \ref{ex:one-sidedKBL-moment} and \ref{ex:one-sidedKBL-moment-II}}\label{B}
\end{subfigure}

\end{tabular}
\caption{\small (A) Solid line: 1101 nodes used in the trapezoid rule, dots: 33 nodes
used  in the Z-SINH-I algorithm.  (B) Solid line: 901 nodes used in the trapezoid rule.
Almost vertical line and dots: 306 and 55 nodes used in the Z-SINH-I and Z-SINH-II algorithm in  
Examples \ref{ex:one-sidedKBL-moment} and \ref{ex:one-sidedKBL-moment-II}, respectively. }
\label{fig:Z_SINH1.pdf}
\end{figure}

}
\end{example}

\begin{example}\label{example:SINH1}{\rm We compare the complexity of the trapezoid rule and sinh-acceleration
algorithm in the following more general situation.
Assume that $\al\in (\pi/2,\pi)$, and the following condition similar 
 to \eq{eq:bound_main_trap} holds:
there exist $C_{\tu}>0$ and $m_\tu$ such that
 \bbe\label{eq:main_bound_2}
|\tu(z)|\le C_{\tu}|1-z|^{-1}|z|^{m_\tu}, \
z\in 1-\cC_{\al}.
\ee
Note that if $n$ is small or moderate, then the sinh-deformation brings small advantages or none. Hence, we assume that $n$ is large. In this case, the q-Hardy norm is very large unless the interval
$[r_-,r_+]:=\chi_{L; \sg_\ell,b_\ell,\om_\ell}(S_{(-d_\ell,d_\ell)})\cap\bR\subset(0,1)$ 
is very close to 1,
and $\chi_{L; \sg_\ell,b_\ell,\om_\ell}(S_{(-d_\ell,d_\ell)})$ is at the distance $r_-$ from the origin.

\vskip0.1cm
\noindent
{\sc Algorithm.}  
Let Condition $Z$-SINH1$(a_-,a_+;\al)$ with $\al\in (\pi/2,\pi)$ and \eq{eq:main_bound_2} hold, and let $n$ be large. Given $M$ and the error tolerance $\eps$,
\begin{enumerate}[(1)]
\item
choose $M_1<M$ close to $M$, e.g., $M_1=0.9M$;
\item
set $r_-=e^{-(M+M_1)/n}$ and $r_+=e^{-(M-M_1)/n}$;
\item
define $\om_\ell=\pi/4-\al/2$ and
$d_\ell=k_d(\al/2-\pi/4)$, where $k_d<1$ is close to 1, e.g., $k_d=0.9$;
\item
define
 $(\sg_\ell,b_\ell)$ by \eq{sgell_sinh_L}, \eq{bell_sinh_L};
\item
define step $\ze_\ell$ and number of terms $N_0=\La/\ze_\ell$ by \eq{rec_ze_ze} and \eq{eqLa_z}.
\end{enumerate}
To apply \eq{rec_ze_ze}, we need an efficient approximation to the q-Hardy norm.
The q-Hardy norm being an integral one, we may derive an approximate bound working in the $z$-plane. On the strength of
\eq{eq:main_bound_2}, $H(f_n,d_\ell)$ admits an approximation 
\[
H(f_n,d_\ell)\approx \frac{1}{\pi} C_{\tu}(r_+^{-n-1}H(r_+)+r_-^{-n-1}H(r_-)), \]
where 
\[
H(r_\pm)=
\int_0^{+\infty}\left|1-r_\pm-te^{i(\om_\ell+\pi/2\mp d_\ell)}\right|^{-1}(1+t)^{-n-1+m_{\tu}}dt.
\]
%Similarly to \eq{eqpm:H_trap},
Straightforward calculations show that $H(r_\pm)\sim -2\ln(1-r_\pm)$ as $n\to \infty$.
If
$r_+$ is chosen not too close to 1,  then 
$H(f_n,d_\ell)\approx (2/\pi) C_{\tu}e^{M+M_1}n/(M+M_1)$,
$\ln H(f_n,d_\ell)\approx 2M+\ln n, $
and 
\bbe\label{ze_ell}
\ze_\ell\approx \frac{2\pi d_\ell}{E+\ln n+ 2M}\approx \frac{k_d\pi(\al-\pi/2)}{E+\ln n+ 2M}.
\ee
Assuming that  $\al\in (\pi/2,\pi)$ is close to $\pi$, we have $\om_\ell+d_\ell\approx 0$, $\om_\ell-d_\ell\approx -\pi/2$, 
\[
b_\ell=\frac{r_+-r_-}{\sin(\om_\ell+d_\ell)-\sin(\om_\ell-d_\ell)}\approx \frac{2M}{n},\] and  \eq{eqLa_z} gives $\La$ satisfying
\bbe\label{eqLa_z_1}
\La\approx C_{\tu}\frac{E}{n-m_{\tu}}+\ln n-\ln M,
\ee
 where $C_{\tV,\ga}$ is independent of $\eps$ and $n$. The number of terms is, approximately,
 \bbe\label{N_ell_sinh}
 N_\ell\approx 2\frac{E+\ln n+ 2M}{k_d\pi^2}\left(C_{\tu}\frac{E}{n-m_{\tu}}+\ln (n/M)+\ga\right).
 \ee
 Assume that  $n$ is large but not extremely large so that $n>>m_\tu$,
 $n>> E>> \ln n$, and 
  $E>>M$. Then we may use the approximation
  \bbe\label{N_ell_sinh_2}
N_\ell\approx2\frac{(E+2M)\ln (n/M)}{k_d\pi^2}.
\ee
}
\end{example}
Comparing \eq{Neps_n_M} with \eq{N_ell_sinh_2}, we see that the complexity of the trapezoid rule
exceeds the complexity of the new numerical realization of the inverse $Z$-transform by the factor 
$K\approx (n/M)/\ln(n/M)$. Thus, 
%although the choice of the deformations can be improved, 
 the sinh-acceleration is more efficient than  trapezoid rule if $n>>M$.

\begin{example}\label{ex:one-sidedKBL-moment}{\rm In Fig.~\ref{fig:Z_SINH1.pdf} (B), we plot nodes used for the evaluation of the 100-th moment of the distribution of a KoBoL subordinator with positive drift (in Example \ref{example:SINH1simple}, the drift is zero). 
The moment generating function is  $M(z)=e^{\Psi(z)}$; $\Psi(z)=\mu z+c\Ga(-\nu)((\la-z)^\nu-\la^\nu)$, $c=0.1, \nu=0.5, \la=1.01$, $\mu=0.05$. $\mu_{100}=5.60408317840114E-05$; the difference between values obtained with the two algorithms is smaller that $E-15$.
Circle: 1101 nodes used in the trapezoid rule. The number is chosen by hand as an approximately minimal number
 which satisfies the error tolerance $\eps=10^{-15}$.
  The almost vertical angle: 306 nodes used in the sinh-acceleration. 56 dots are nodes used in
  the modification of Z-SINH algorithm constructed in the following section - see Example \ref{ex:one-sidedKBL-moment-II}.
Since $\mu>0$, $M(z)$ decreases as $z\to\infty$ in the left half-plane but increases as $z\to\infty$ along any ray in the right half-plane. Hence, we may choose only $\om_\ell>0$ and $d_\ell>0$ so that
$\om_\ell-d_\ell\ge 0$. 
We use $r_-=0.98$ and $r_+=1$ close to 1, and the approximate recommendations  at the end of Section \ref{param_choice_Z_SINH}. Explicitly, $\sg_\ell=1.005$, $b_\ell=0.245$, $\om_\ell= 0.0612$,
$d_\ell=0.0408$, $\ze_\ell=0.0069$ (rounded). The prescription \eq{eqLa_z} gives unnecessary large $\La$, which we decrease by the factor 0.8.
 CPU times: 143 and 87 microsec. for the trapezoid rule and sinh-acceleration, respectively, the average over 100,000 runs.
}
\end{example}

\section{Sinh-acceleration II and III and Log-acceleration}\label{s:sinh_II} 
\subsection{Sinh-acceleration II}\label{s:sinh2}
If Condition $Z$-SINH1$(a_-,a_+;\al)$ holds with $\al<\pi/2$ and $\pi/2-\al$ is not small,  then the construction in Section \ref{s:sinh_I} gives a small $d_\ell$, hence, small $\ze_\ell$. At the same time, 
the truncation parameter $\La$ is not small. Hence, the number of terms $N_\ell$ is large. 
We alleviate these difficulties imposing an analog of Condition $Z$-SINH1$(a_-,a_+;\al)$ in terms of the function
$\tv(w):=\tu(w^2)$, and modifying the construction of the deformation as follows.

\vskip0.1cm
\noindent
{\sc Condition $Z$-SINH2$(a_-,a_+;\al)$.}  
 There exist $0\le a_-<1\le a_+$ and  $\al\in (\pi/2,\pi)$  such that
 \begin{enumerate}[(a)]
 \item
 $\tu$ admits analytic continuation to $\cD(a_-^2, a_+^2)$;
 \item $\tv(w)$ admits analytic continuation
 to $\cU^+(a_-,a_+,\al):=\{w\in \cU(a_-,a_+,\al)\ |\  \Re\,w\ge 0\}$;
 \item
 for any $a_-<r_-<r_+<a_+$ and $\al_1\in (\pi/2,\al)$,
  \bbe\label{eq:main_bound_3}
|\tv(w)|\le C_\tv(r_-,r_+;\al_1)(1+|w|)^{2m_\tu}, \
z\in \cU^+(r_-, r_+;\al_1),
\ee
where $m_\tv$ depends only on $\tv$, and $C_\tv(r_-,r_+;\al_1)$ depends on $(r_-,r_+,\al_1)$.
\end{enumerate} 
We change the variable $z=w^2$ in \eq{eq:invZ}
\bbe\label{eq:invZ_2}
u_n=\frac{1}{\pi i}\int_{|w|=1, \Re\,w\ge 0}\tv(w)w^{-2n-1}dw,
\ee
and note that $\tv(-w)(-w)^{-2n-1}=-\tv(w)w^{-2n-1}$  for $w\in i(-\infty, -1]\cup i[1,+\infty)$.
Hence,
\bbe\label{eq:invZ_2b}
u_n=\frac{1}{\pi i}\int_{\cL_0}\tv(w)w^{-2n-1}dw,
\ee
where $\cL_0=i(-\infty,-1]\cup\{w\ |\ |w|=1, \Re\,w>0\}\cup i[1,+\infty)$.

The  following theorem is a straightforward
modification of Theorem \ref{thm:sinhI} (c), (d).

\begin{thm}\label{thm:sinhII} Let Condition $Z$-SINH2$(a_-,a_+;\al)$ hold and $n>m_\tv$. Then
\begin{enumerate}[(a)]
\item
 for any $\om_\ell
\in (-\al/2, 0]$, there exist $b_\ell>0$ and $\sg_\ell\in (a_-+b_\ell\sin(\om_\ell), a_++b_\ell\sin(\om_\ell))$
s.t.
\bbe\label{eq:invZ_2c}
u_n=\frac{1}{\pi i}\int_{\cL_{L,\sg_\ell,b_\ell,\om_\ell}}\tv(w)w^{-2n-1}dw;
\ee
\item
 there exists $d_\ell>0$ such that $\chi_{L;\sg_\ell, b_\ell,\om_\ell}(S_{(-d_\ell,d_\ell)})\subset\cU^+(a_-,a_+,\al)$.
\end{enumerate}
\end{thm}
The sinh-change of variables, choice of parameters and error bounds are essentially as in Section \ref{s:sinh_I}. The
differences are: 1) only $\om_\ell\in (-\al/2, 0)$ are inadmissible, and $d_\ell\in (0, \min\{-\om_\ell, \om_\ell+\al/2\})$; an approximately optimal choice is
$\om_\ell=\pi/4-\al/2$ and $d_\ell=-k_d\om_\ell$, where $k_d<1$. Hence, $\ze_\ell$ cannot be made as
large as in Section \ref{s:sinh_I}; 
2) in an approximation for the q-Hardy norm,  $2n$ should be used instead of $n$, hence, the truncation parameter is smaller, 3) a choice
of $r_\pm$ sufficiently close to 1 is determined by $2n$ instead of $n$  (the maximum of
$|z|^{-n-1}$ over $\chi_{L;\sg_\ell, b_\ell,\om_\ell}(S_{(-d_\ell,d_\ell)})$ should be not too large to avoid large rounding errors).

\begin{example}\label{ex:one-sidedKBL-moment-II}{\rm   In Example \ref{ex:one-sidedKBL-moment}, $\mu_{100}$ is calculated with the accuracy better than E-15 using 
Z-SINH-II algorithm with $N_\ell=55$, the CPU time is 19 microsec., the average over 100,000 runs.
See Fig.~\ref{fig:Z_SINH1.pdf} (B) for illustration.}
\end{example}
\begin{example}\label{ex:mixture_one-sidedKBL-delta-II}{\rm  The moment generating
function is $0.3e^{\mu z}+0.7M(z)$, where $\mu=2$ and $M(z)$ is from Example \ref{example:SINH1simple}. 
$\mu_{100}=3.72680559839856E-05$ is calculated with the accuracy better than E-15 using the trapezoid rule,
 Z-SINH-I and
Z-SINH-II algorithms with $N=1101$,  $N_\ell=306$, $N_\ell=  56$, respectively.
 The CPU times are 183, 109 and 23 microsec., respectively. %, the average over 100,000 runs.
}
\end{example}
In some cases, it is useful or even necessary to make a further change of variables in \eq{eq:invZ_2b}, 
\bbe\label{eq:invZ_2d0}
u_n=\frac{p}{\pi i}\int_{\cL^{1/p}_0}\tv(w_1^p)w_1^{-2np-1}dw_1,
\ee
where $p>1$ and $\cL^{1/p}_0:=\{w_1\ |\ \Re\, w_1>0, w_1^p\in \cL_0\}$, and then apply the sinh-acceleration.
The  following theorem is a straightforward
modification of Theorem \ref{thm:sinhII}.

\begin{thm}\label{thm:sinhIIb} Let Condition $Z$-SINH2$(a_-,a_+;\al)$ hold and $n>m_\tv$. Then
\begin{enumerate}[(a)]
\item
 for any $p>1$, $\om_\ell
\in ((\pi-\al)/(2p)-\pi/2, \pi/(2p)-\pi/2)$, there exist $b_\ell>0$ and 
 $\sg_\ell\in ((a_-)^{1/p}+b_\ell\sin(\om_\ell), (a_+)^{1/p}+b_\ell\sin(\om_\ell))$ such that
\bbe\label{eq:invZ_2d}
u_n=\frac{p}{\pi i}\int_{\cL_{L,\sg_\ell,b_\ell,\om_\ell}}\tv(w_1^p)w_1^{-2np-1}dw_1;
\ee
\item
 there exists $d_\ell>0$ such that $\{z^p\ |\ \Re\,z>0, z\in \chi_{L;\sg_\ell, b_\ell,\om_\ell}(S_{(-d_\ell,d_\ell)})\}\subset\cU^+(a_-,a_+,\al)$.
\end{enumerate}
\end{thm}
An approximately optimal choice is $\om_\ell=(\gap+\gam)/2$, $d_\ell=k_d(\gap-\gam)/2$, where $k_d=0.9$;
$r_-<r_+$ are chosen so that $(r_\pm)^p\in (a_-,a_+)$ are close to 1, and  $(\sg_\ell,b_\ell)$ are defined by \eq{sgell_sinh_L}, \eq{bell_sinh_L}.  In an approximation for the q-Hardy norm,  $2np$ should be used instead of $n$.

\subsection{Sinh-acceleration III}\label{s:sinh3}
Conditions $Z$-SINH1 and $Z$-SINH2 are too restrictive if the moments of distributions not on $\bR_+$ but on $\bR$ are evaluated. To cover this case,
we rewrite \eq{eq:invZ} as
%\beqast
%u_n&=&\frac{1}{2\pi i}\left(\int_{|z|=1,\Re\,z>0}\tu(z)z^{-n-1}dz +
%\int_{|z|=1,\Re\,z<0}\tu(z)z^{-n-1}dz\right)\\
%&=&\frac{1}{2\pi i}\left(\int_{|z|=1,\Re\,z>0}\tu(z)z^{-n-1}dz -
%\int_{|z|=1,\Re\,z>0}\tu(-z)(-z)^{-n-1}dz\right).
%\eqast
%Simplifying, 
\bbe\label{eq:invZ2}
u_n=\frac{1}{2\pi i}\int_{|z|=1,\Re\,z>0}(\tu(z)+(-1)^n\tu(-z))z^{-n-1}dz.
\ee
Assume that $\tu$ admits analytic continuation to an open domain $\cU$ in the right half-plane,
 containing $i(\bR\setminus (-1,1))\cup\{z=e^{i\phi}\ |\ \phi\in (-\pi/2-\eps,-\pi/2+\eps)\cup (\pi/2-\eps,\pi/2+\eps)\}$,
where $\eps>0$. For $z\in \cU$, we have
%\[
$ (\tu(-z)+(-1)^n\tu(-(-z)))(-z)^{-n-1}=-(\tu(-z)(-1)^n+\tu(z))z^{-n-1},$
%\]
therefore, 
we may rewrite \eq{eq:invZ4} as follows
\bbe\label{eq:invZ4}
u_n=\frac{1}{2\pi i}\int_{\cL_0}(\tu(z)+(-1)^n\tu(-z))z^{-n-1}dz,
\ee
where $\cL_0=i(-\infty,-1]\cup\{z\ |\ |z|=1, \Re\,z>0\}\cup i[1,+\infty)$.
If the domain of analyticity of $\tu$ is sufficiently large, we can deform $\cL_0$ into an appropriate
sinh-contour $\cL_{L,\sg_\ell,b_\ell,\om_\ell}$.
A sufficient simple condition is 
the following analog of Condition $Z$-SINH1$(a_-,a_+;\al)$. 
\vskip0.1cm
\noindent
{\sc Condition $Z$-SINH3$(a_-,a_+;\ga)$.}  
 There exist   $0<\ga\le \pi/2$, and $0\le a_-<1\le a_+$ such that
 \begin{enumerate}[(a)]
 \item
  $\tu$ admits analytic continuation to
\[
 \cU^{\mathrm{sym}}(a_-,a_+,\ga):=
((-a_+,a_+)+i(\cC_\ga\cup\{0\}\cup(-\cC_\ga)))\setminus \{z\ |\ |z|\le a_-\};
\]
 \item for any $a_-<r_-<r_+<a_+$ and $\ga'\in (0,\ga)$,
  \bbe\label{eq:main_bound_III}
|\tu(z)|\le C_\tu(r_-,r_+;\ga')(1+|z|)^{m_\tu}, \
z\in \cU^{\mathrm{sym}}(r_-, r_+;\ga'),
\ee
where $m_\tu$ depends only on $\tu$, and $C_\tu(r_-,r_+;\ga')$ depends on $(r_-,r_+,\ga')$.
\end{enumerate}
\begin{thm}\label{thm:sinhIII} Let Condition $Z$-SINH3$(a_-,a_+;\ga)$ hold and $n>m_\tv$. Then
\begin{enumerate}[(a)]
\item
 for any $\om_\ell
\in (-\ga/2, \ga/2)$, there exist $b_\ell>0$ and $\sg_\ell\in (a_-+b_\ell\sin(\om_\ell), a_++b_\ell\sin(\om_\ell))$
such that
\bbe\label{eq:invZ5}
u_n=\frac{1}{2\pi i}\int_{\cL_{L,\sg_\ell,b_\ell,\om_\ell}}(\tu(z)+(-1)^n\tu(-z))z^{-n-1}dz;
\ee
\item
there exists $d_\ell>0$ such that $\chi_{L;\sg_\ell, b_\ell,\om_\ell}(S_{(-d_\ell,d_\ell)})\subset\cU^{\mathrm{sym}}(a_-,a_+,\ga)$.
\end{enumerate}
\end{thm}
The choice of parameters and error bounds are essentially as in Section \ref{s:sinh_I}. The main difference is that 
 only $\om_\ell\in (-\ga,\ga)$ are admissible, and an approximately optimal choice is
$\om_\ell=0$ and $d_\ell=k_d\ga$, where $k_d<1$ is close to 1.

\begin{example}\label{example:SINH3_1}{\rm Changing the order $\nu=1.5$ of KoBoL in Example \ref{example:SINH1simple}: $M(z)=e^{\Psi(z)}$; $\Psi(z)=c\Ga(-\nu)((\la-z)^\nu-\la^\nu)$, $c=0.1, \nu=1.5, \la=1.01$, 
we obtain a function which does not satisfy Condition $Z$-SINH1 but  satisfies Condition $Z$-SINH3$(a_-,a_+;\ga)$ with 
$\ga=(\pi/2)\min\{1-1/\nu, 3/\nu-1\}=\pi/6$
(see \cite{EfficientAmenable} for the calculation of the maximal cone of analyticity of $M(z)$, where $M(z)$ is bounded), $m_\tu=0$, $a_+=1.01$, and $a_-=0$.
 $\mu_{100}=3.00859241487316E-07$; the difference between values obtained with the two algorithms is smaller that $E-15$. The number of nodes: 1001 and 72, 
 CPU times: 159 and 25 microsec. for the trapezoid rule and sinh-acceleration, respectively, the average over 100,000 runs.
 }
\end{example}

\begin{example}\label{example:SINH3_2}{\rm Let $M(z)=e^{\Psi(z)}$, where $\Psi(z)=\de(\la^\nu-(\la^2-z^2)^{\nu/2})$, $\nu\in (0,2)$, $\la>1$, $\de>0$,
be the Laplace exponent of a symmetric Normal Tempered Stable (NTS) process \cite{B-N-L}. Then
Condition $Z$-SINH3$(a_-,a_+;\ga)$ is satisfied with 
$\ga=(\pi/2)\min\{1/\nu, 1\}$, $0=a_-<a_+\le \la$. Non-symmetric case is treated similarly but if $\nu\in (0,1)$ and the drift $\mu\neq 0$ is introduced then $Z$-SINH3$(a_-,a_+;\ga)$ fails.
}
\end{example}

\begin{rem}\label{rem:SINH3}{\rm 
\begin{enumerate}[(1)]
\item
In some cases, it is useful to change the variable $z=w^p$, where $p>1$:
\bbe\label{eq:invZ5p}
u_n=\frac{p}{2\pi i}\int_{\cL_0}(\tu(w^p)+(-1)^n\tu(-w^p))w^{-np-1}dw,
\ee
and then use a sinh-deformation with $\om_\ell\in (\gam,\gap)$, where $\gap=\pi/2(1/p-1)$, $\gam=\gap-\ga/p$.
An approximately optimal choice is $\om_\ell=(\gap+\gam)/2, d_\ell=k_d(\gap-\gam)/2$. 
\item
It may be useful to start with the rotation of  the complex plane $z=z'e^{i\varphi}$ in \eq{eq:invZ} so that
$\tv(z')=\tu(z'e^{i\varphi})$ satisfies Condition $Z$-SINH3$(a_-,a_+;\ga)$ with a larger $\ga$, and then apply \eq{eq:invZ4} or \eq{eq:invZ5p}.
\end{enumerate}
}
\end{rem}

\begin{example}\label{example: SINH3 fractional}{\rm 
Let $\tu$ be a rational function without poles on $\bT$, and let $\cZ$ be the set of poles outside $\bT$.
We choose $\varphi\in [-\pi,\pi)$ so that the angular distance from $\{e^{i\varphi},e^{-i\varphi}\}$ to $\cZ$ is maximal.
} 
\end{example}
\subsection{Log-acceleration}\label{ss:Log}
If $\tu(z)$ is a linear combination of functions satisfying Condition $Z$-SINH3 with $\al>\pi/2$, exponentials functions $e^{xz}$ with $x\in \bR\setminus\{0\}$ (this is the case if
$\tu$ is $Z$-transform of a measure having atoms or moment generating function of probability distributions of wide classes of L\'evy processes of finite variation with non-zero drift) then Condition $Z$-SINH3$(a,b;\gam,\gap)$ fails
but a weaker form of this condition is satisfied.

\vskip0.1cm
\noindent
{\sc Condition $Z$-LOG$(a,b;\al)$.}  
 There exist   $\al>0$  and $0\le a_-<1\le a_+$ such that
 \begin{enumerate}[(a)]
 \item $\tu(z)$ admits analytic continuation
 to \[
 \cU^{\mathrm{log}}(a_-,a_+,\al):=
 \{z\ |\ (a_-<|z|<a_+) \vee 
 (|z|\ge a_+, |\Re\,z|\le a_++\al\ln (1+|\Im\,z|))\};
 \]
 \item
 for any $a_-<r_-<r_+<a_+$ and $\al'\in (0,\al)$,
  \bbe\label{eq:main_bound_LOG}
|\tu(z)|\le C_\tu(r_-,r_+;\al')(1+|z|)^{m_\tu+m'_\tu \al'}, \
z\in \cU^{\mathrm{log}}(r_-, r_+;\al'),
\ee
where $m_\tu, \mu'_\tu$ depend only on $\tu$, and $C_\tu(r_-,r_+;\al')$ depends on $(r_-,r_+,\al')$.
\end{enumerate} 
We start with the reduction to \eq{eq:invZ4}, and then,
instead of the function $\chi_{L;\sg_\ell,b_\ell,\om_\ell}$ defining the sinh-deformation and sinh-change of variables, we
use the function
\bbe\label{chi_log}
\chi_{\mathrm{log}; \sg_\ell, A}(y)=\sg_\ell+iy\ln(A+y^2),
\ee
where $A>1$, and the contour (vertical line) $\cL_{\mathrm{log}; \sg_\ell, A}=\chi_{\mathrm{log}; \sg_\ell, A}(\bR)$.
Making the change of variables $z=z(y):=\chi_{\mathrm{log}; \sg_\ell, A}(y)$, we obtain
\bbe\label{eq:invZ7}
u_n=\int_{\bR}\frac{1}{2\pi}(\tu(z(y))+(-1)^n\tu(-z(y)))z(y)^{-n-1}\left(\ln(A+y^2)+\frac{2y^2}{A+y^2}\right)dy.
\ee
Denote by $f_n$ the integrand on the RHS of \eq{eq:invZ7}. The restrictions on $d_\ell$, the half-width of a strip of analyticity
of $f_n$ around $\bR$, are 
$
a_-<\sg_\ell-d_\ell\ln (A-d_\ell^2), \sg_\ell+d_\ell\ln (A-d_\ell^2)<a_+,
$
and $\chi_{\mathrm{log}; \sg_\ell, A}(S_{(-d_\ell,d_\ell)})\subset \cU^{\mathrm{log}}(a_-,a_+,\al)$. It is easy to see
that if $R$ is sufficiently large and $\sg_\ell+d_\ell\ln (A-d_\ell^2)<a_+$, then 
$\chi_{\mathrm{log}; \sg_\ell, A}(S_{(-d_\ell,d_\ell)})\cap\{z\ |\ (|z|>R) \vee (|\Im\, z|\le 1/R)\}\subset \cU^{\mathrm{log}}(a_-,a_+,\al)\cap\{z\ |\ (|z|>R) \vee (|\Im\, z|\le 1/R)\}$. 
The complete verification is rather involved but, evidently,  
$\chi_{\mathrm{log}; \sg_\ell, A}(S_{(-d_\ell,d_\ell)})
\subset \cU^{\mathrm{log}}(a_-,a_+,\al)$ 
if $d_\ell>0$ is sufficiently small. 
We are interested in the case of $r_\pm$ close to 1, therefore, we choose 
$A$ so that $(A-1)/(r_+-r_-)>>1$,
$\sg_\ell=(r_++r_-)/2$, and find $d_\ell$ from
$\sg_\ell+d_\ell\ln (A-d_\ell^2)=r_+$. Then, as $r_+-r_-\to 0$, $d_\ell\sim (r_+-r_-)/(2\ln A)$. Since the discretization error of the infinite trapezoid rule decreases as $d_\ell$ increases, we choose $A$ close to 1, e.g., $A=1+(r_+-r_-)^{1/4}$.
Then $d_\ell$ is of the order of $(r_+-r_-)^{3/4}$. The integrand decays as $(y\ln y)^{m_\tu-n}y^{-1}$, therefore, if $n>>m_\tu$, the truncations parameter $\La$ is not large. 
The q-Hardy norm is of the order of $r_-^{-n}$, as the Hardy norm in the case of the trapezoid rule is, therefore, given the error tolerance,
the number of terms of the simplified trapezoid rule is, approximately, $(r_+-r_-)^{-1/4}$ smaller than the number of terms
of the trapezoid rule.

\begin{example}\label{LOG_1}{\rm Let $M(z)=e^{\Psi(z)}$, where $\Psi(z)=\mu z+\de(\la^\nu-(\la^2-z^2)^{\nu/2})$, $\nu\in (0,1)$, $\la>1$, $\de>0$, and $\mu\neq 0$. Then Condition $Z$-SINH3$(a_-,a_+;\al)$ fails but $Z$-LOG$(a,b;\ga)$ is satisfied with 
any $\al>0$, $m'_\tu=|\mu|$, $m_\tu=0$. 
$\mu_{100}=6.16741619667409E-05$, the error tolerance smaller than E-15 is achieved using the trapezoid rule and log-acceleration with 900 and 102 nodes, the CPU times are 221 and 47 microsec., respectively, the average over 100,000 runs.

The results are similar for  $M(z)$  as in Example \ref{ex:one-sidedKBL-moment} but with negative drift $\mu<0$.}
\end{example}

\begin{rem}\label{rem:log}{\rm The map \eq{chi_log} and corresponding change of variables is a special case
of a more general class of the log-acceleration family of deformations in \cite{iFT,ConfAccelerationStable}. In the current setting, it can be advantageous to use contours deformed into the right half-plane.
}
\end{rem}

\section{Inverse $Z$-transform, Wiener-Hopf factorization of functions on $\bT$,
and construction of causal filters}\label{s:ZWHFfilter}
Consider the problem of calculation of the impulse response $h[n], n=0,1,\ldots,$ $h[n]\in \bR$, of a linear translation invariant filter
given the power spectral density $PSD(z)$. For $z\in \bT$, $PSD(z)=H(z)H(1/z)$, where 
$
H(z)=\sum_{n=0}^\infty h[n]z^{-n}$ is the transfer function of the filter. In order to apply the results of the preceding sections,
we need to impose conditions on $PSD(z)$ which ensure that $\tu(z):=H(1/z)$ satisfies one of the conditions $Z$-SINH1 - $Z$-SINH3 or $Z$-LOG. We formulate conditions on $PSD(z)$, which allows us to prove that
$\tu(z)$ satisfies Condition $Z$-SINH3, and apply SINH-III algorithm to calculate the impulse response. 
The Wiener-Hopf factorization can be done using the trapezoid rule or, more efficiently, the sinh-acceleration.

\vskip0.1cm
\noindent
{\sc Condition WHF-SINH3$(a;\ga;m_+,m_-)$.}  
  There exist $0<\ga\le \pi/2$, $a>1$, $c_{\infty}>0$, $m_\pm\in \bR$ and $\de\in (0,1]$ such that
  \begin{enumerate}[(1)]
  \item
  $PSD(z)>0, z\in \bT$;
  \item
  $PSD$ is analytic in %admits analytic continuation to
 %\[
$ \cU_{WHF}(a,\ga):= \cU^{\mathrm{sym}}(-a,a,\ga)\cup\{z\ |\ 1/z\in \cU^{\mathrm{sym}}(-a,a,\ga)\};$
 %\]
 \item
for any $z\in \cU_{WHF}(a,\ga)$, $PSD(1/z)=PSD(z)$ and
 $PSD(z)\not\in (-\infty,0]$;
 \item
  as $(\cU_{WHF}(a,\ga)\cap\{\pm\Re\, z>0\}\ni) z\to \infty$, $PSD(z)$ admits the representation
 \bbe\label{T(z)p}
 PSD(z)=c_{\infty} e^{\mp i\pi m_\pm}|z|^{m_++m_-}(1+B_\infty(z)), 
 \ee
 where $B_\infty(z)$ admits the following bound: for any $1<r<a$ and $\ga'\in (0,\ga)$,
  \bbe\label{eq:main_bound_WHFinf}
|B_\infty(z)|\le C(r;\ga')|z|^{-\de};
\ee
the constant $C(r;\ga')$ depends on $r,\ga'$ but not on $z\in \cU_{WHF}(a,\ga)\cap\{z\ |\  |z|>1/r\}$;
\item
 as $(\cU_{WHF}(a,\ga\cap\{\pm\Re\, z>0\})\ni)z\to 0$, $PSD(z)$ admits the representation
 \bbe\label{T(z)m}
 PSD(z)=c_{\infty} e^{\mp i\pi m_\pm}|z|^{-(m_++m_-)}(1+B_0(z)), 
 \ee
 where $B_0(z)$ admits the following bound: for any $1<r<a$ and $\ga'\in (0,\ga)$,
  \bbe\label{eq:main_bound_WHF0}
|B_0(z)|\le C(r;\ga')|z|^{\de};
\ee 
the constant $C(r;\ga')$ depends on $r,\ga'$ but not on $z\in\cU_{WHF}(a,\ga)\cap\{z\ |\ |z|\le r\}$.
\begin{example}\label{ex_filtering_1}{\rm 
$PSD(z)=(a_+-z)^{m_+}(a_+-1/z)^{m_+}(a_-+z)^{m_-}(a_-+1/z)^{m_-}$, where $a_\pm>0$, satisfies
Condition WHF-SINH3$(a;\pi/2;m_+,m_-)$ with $a=\min\{a_+,a_-\}$, $c_\infty=a^{m_++m_-}$ and $\de=1$.
}
\end{example} 
\end{enumerate}
\subsection{Wiener-Hopf factorization}\label{ss:WHF-filtering}
Conditions (1)-(5) allow us to simplify the standard construction of the Wiener-Hopf factors and derive 
modifications of the formulas for the Wiener-Hopf factors amenable for efficient calculations. Introduce functions
\beqa\label{def_A}
A(z)&=&\frac{a^{m_++m_-}}{c_\infty (a-z)^{m_+}(a-1/z)^{m_+}(a+z)^{m_-}(a+1/z)^{m_-}}PSD(z),\\
\label{def_Apm}
A_\pm(z)&=&\exp\left[\pm \frac{1}{2\pi i}\int_{|z'|=r^{\pm 1}}\frac{\ln A(z')}{z-z'}dz'\right],
\eqa
and define constants
\bbe\label{def_d}
d=-\frac{1}{2\pi i}\int_{|z'|=1}\frac{\ln A(z')}{z'}dz'
\ee
and $c_\pm=c_\infty^{1/2}a^{-(m_++m_-)/2}e^{\pm d/2}$. Then introduce
\beqa\label{def_Hp}
H_+(z)&=&c_+(a-z)^{m_+}(a+z)^{m_-}A_+(z), \\\label{def_Hm}
 H_-(z)&=&c_-(a-1/z)^{m_+}(a+1/z)^{m_-}A_-(z).
\eqa
Let  the curves $\cL_{L;\sg_+,b_+,\om_+}$ and $-\cL_{L;\sg_-,b_-,\om_-}$ be subsets of $\cU_{WHF}(a,\ga)\cap\{z\ |\ \Re\, z\ \pm 0\}$. On the former curve, the direction is up, on the latter - down. On each of the curves
 $1/\cL_{L;\sg_\pm ,b_\pm ,\om_\pm}=\{z\ |\ 1/z\in 1/\cL_{L;\sg_\pm,b_\pm,\om_+}\}$ and 
 $-1/\cL_{L;\sg_\pm ,b_\pm ,\om_\pm}$, the direction is anti-clockwise. 
\begin{thm}\label{thm:Apm_analcont}\begin{enumerate}[(a)]
\item
$A_\pm$ and $H_\pm$ admit analytic continuation to $\cU_{WHF}(a,\ga)$;
\item  
for $z\in \cU_{WHF}(a,\ga)$,
\bbe\label{WHF_disc}
H_+(z)H_-(z)=PSD(z);
\ee
 \item
for any $z\in \cU_{WHF}(a,\ga)\cup\{|z|<a\}$ between $\cL_{\sg_+,b_+,\om_+}$ and $-\cL_{\sg_-,b_-,\om_-}$,
\bbe\label{WHFp}
\ln A_+ (z)= -\frac{1}{2\pi i}\left(\int_{\cL_{L;\sg_+,b_+,\om_+}}+\int_{-\cL_{L;\sg_-,b_-,\om_-}}\right)\frac{\ln A(z')dz'}{z-z'},
\ee
and
 for any $z\in \cU_{WHF}(a,\ga)\cup\{|z|>1/a\}$ in the exterior of the union of the regions bounded by $1/\cL_{L;\sg_+,b_+,\om_+}$ and $-1/\cL_{L;\sg_-,b_-,\om_-}$, 
\beqa\label{WHFm}
\ln A_- (z)&=& \frac{1}{2\pi i}\left(\int_{1/\cL_{L;\sg_+,b_+,\om_+}}+\int_{-1/\cL_{L;\sg_-,b_-,\om_-}}\right)\frac{\ln A(z')dz'}{z-z'};
\eqa
\item  
for $z\in \cU_{WHF}(a,\ga)$, $H_-(1/z)=H_+(z)$;

\item
for any $1<r<a$, $\ga'\in (0,\ga)$, and $\de'\in (0,\de)$, 
\bbe\label{Apm}
A_\pm(z)-1=O((|z|+1/|z|)^{-\de'}), \quad (\cU_{WHF}(r,\ga')\ni)
z\to \{0,\infty\}. 
\ee
\end{enumerate}
\end{thm}
\begin{proof} (a) It suffices to consider $A_\pm$. Clearly, $A_+$ (resp., $A_-$) is analytic in $\{z\ |\ |z|<a\}$ (resp., $\{z\ |\ |z|>1/a\}$), and since $A(z)$ admits analytic continuation to $\cU_{WHF}(a,\ga)$, we can define analytic continuation of $A_-$ to $\cU_{WHF}(a,\ga)\cap \{z\ |\ |z|\le 1/a\}$ by $A_-(z)=A(z)/A_+(z)$. Analytic continuation of $A_+$ is by symmetry.

(b) For any $r\in (1,a)$, the function $A(z)$ is analytic in the closed
annulus $\overline{\cD(1/r,r)}$, therefore, by the Cauchy integral theorem, for any $z\in \cD(1/r,r)$,
\bbe\label{lnA}
\ln A(z)=-\frac{1}{2\pi i}\int_{|z'|=1/r}\frac{\ln A(z')}{z-z'}dz +\frac{1}{2\pi i}\int_{|z'|=r}\frac{\ln A(z')}{z-z'}dz.
\ee
Since $r\in (1,a)$ is arbitrary, $A(z)=A_+(z)A_-(z)$ for $z\in \cD(1/a,a)$.   On the strength of (a),
$A(z)=A_+(z)A_-(z)$ for $z\in \cU_{WHF}(a,\ga)$, and \eq{WHF_disc} follows.

(c) 
On the strength of  \eq{T(z)p}-\eq{eq:main_bound_WHF0}, for any $r\in (1,a), \ga'\in (0,\ga)$,
\bbe\label{asAz}
\ln A(z)=O((|z|+1/|z|)^{-\de}),\quad (\cU_{WHF}(r,\ga')\ni )z\to \{0,\infty\}.
\ee
It follows that we can deform the contour of integration $\{z\ |\ |z|=a, \Re\,z>0\}\cup
\{z\ |\ |z|=a, \Re\,z<0\}$ in the formula for $A_+(z)$
into the union of contours $\cL_{\sg_+,b_+,\om_+}$ and $-\cL_{\sg_-,b_-,\om_-}$, and obtain \eq{WHFp}.
The proof of \eq{WHFm} is by symmetry (use (b)).

(d) We make the changes of variables $z\mapsto 1/z$ and $z'\mapsto 1/z'$ 
in \eq{WHFm}. Since $A(1/z')=A(z')$ and 
$
-\frac{1}{(1/z-1/z')z'^2}=\frac{z}{(z-z')z'}=\frac{1}{z'}+\frac{1}{z-z'},
$
we derive  
\[
\ln A_- (1/z)=-\frac{1}{2\pi i}\left(\int_{\cL_{\sg_+,b_+,\om_+}}+\int_{-\cL_{\sg_-,b_-,\om_-}}\right)\frac{\ln A(z')dz'}{z-z'}+d=\ln A_+(z)+d,
\]
where
\bbe\label{ds}
d=-\frac{1}{2\pi i}\left(\int_{\cL_{\sg_+,b_+,\om_+}}+\int_{-\cL_{\sg_-,b_-,\om_-}}\right)\frac{\ln A(z')dz'}{z'}
=-\frac{1}{2\pi i}\int_{|z'|=1}\frac{\ln A(z')}{z'}dz'.
\ee
Hence, \[
H_-(1/z)=c_-(a-z)^{m_+}(a+z)^{m_-}e^dA_-(1/z)=c_+(a-z)^{m_+}(a+z)^{m_-}A_+(z)=H_+(z).
\]

(e) For $r\in (1,a)$ and $\ga'\in (0,\ga)$, we take $\om_\pm\in (-\ga,-\ga')$, and then choose $\sg_\pm,b_\pm$ so that the curves $\cL_{L;\sg_\pm,b_\pm,\om_\pm}$ are subsets of $\cU_{WHF}(a,\ga)$ and the right boundary of $\cU_{WHF}(r,\ga')$ is to the left of $\cL_{L;\sg_\pm,b_\pm,\om_\pm}$. Then $\cU_{WHF}(r,\ga')$
is sandwiched between $\cL_{L;\sg_+,b_+,\om_+}$ and $-\cL_{L;\sg_-,b_-,\om_-}$. Furthermore, since
$\om_\pm\in (\ga',\ga)$, the integrand in the formula \eq{WHFp} is bounded away from 0 by
$c(|z|+|z'|)$, where $c$ is independent of $z\in \cU_{WHF}(r,\ga')$ and 
$z'\in \cL_{L;\sg_+,b_+,\om_+}\cup(-\cL_{L;\sg_-,b_-,\om_-})$. Using \eq{asAz}, we obtain that for any $\de'\in (0,\de)$, the integrand on the RHS of \eq{WHFp} admits the bound via $C(1+|z|)^{-\de'}(1+|z'|)^{-1-\de+\de'}$,
where $C$ is independent of $z$ and $z'$. Hence, as $z\to\infty$, $A_+(z)=1+O(|z|^{-\de'})$, and in view of
\eq{asAz} and \eq{lnA}, $A_-(z)=1+O(|z|^{-\de'})$ as well. By symmetry, $A_\pm(z)=1+O(|z|^{\de'})$ as $z\to 0$.

\end{proof}
In the next subsection, we use \eq{WHFm} with $\sg_\pm=\sg, b_\pm=b, \om_\pm=\om$. Changing the variable
$z'= 1/\chi_{L;\sg,b,\om}(y)$  and $z'= -1/\chi_{L;\sg,b,\om}(y)$ in the first and second integral, respectively, and letting $\chi(y)=\chi_{L;\sg,b,\om}(y)$, we obtain
\bbe\label{WHFm2}
\ln A_- (z)=\frac{b}{2\pi }\int_{-\infty}^{+\infty}dy\, \cosh(i\om+y)\left[\frac{\ln A(\chi(y))}{z\chi(y)-1}-\frac{\ln A(-\chi(y))}{z\chi(y)+1}\right].
\ee

\subsection{Calculation of the impulse response function}\label{ss:calc_imp_reponse}
Let $a,r,\ga$ and $m_\pm$ be as above.  For $n\ge 0$ 
\[
h[n]=\frac{1}{2 \pi i}\int_{|z|=1/r}H_-(z)z^{n-1}dz=\frac{1}{2 \pi i}\int_{|z|=r}H_+(z)z^{-n-1}dz
=\frac{1}{2 \pi i}\int_{|z|=r}\frac{PSD(z)}{H_-(z)}z^{-n-1}dz.
\]
Set $\tu(z)=PSD(z)/H_-(z)$, deform the contour above using $Z$-SINH3 algorithm:
\bbe\label{hnsinh}
h[n]=\frac{1}{2\pi i}\int_{\cL_{L,\sg_\ell,b_\ell,\om_\ell}}(\tu(z)+(-1)^n\tu(-z))z^{-n-1}dz.
\ee
It follows from \eq{def_Hp}-\eq{def_Hm} and \eq{Apm} that $\tu$ satisfies Condition $Z$-SINH3$(1/a,a,\ga)$ with $m_\tu=m:=m_++m_-$,
hence, if $n>m$, the deformation is justified. We change the variable $z=\chi(y):=\chi_{L,\sg_\ell,b_\ell,\om_\ell}(y)$:
\bbe\label{hn}
h[n]=\frac{b_\ell}{2\pi}\int_{-\infty}^{+\infty}dy\, \cosh(i\om+y)(\tu(\chi(y))+(-1)^n\tu(-\chi(y)))\chi(y)^{-n-1},
\ee
and apply the simplified trapezoid rule:
\beqa\label{hn_simp_trap}
h[n]&\approx& \frac{b\ze}{2\pi}\sum_{ |j|\le N}\cosh(i\om+j\ze)
%\\\nonumber&&\cdot
(\tu(\chi(j\ze))\chi(j\ze)^{-n-1}+\tu(-\chi(j\ze))(-1)^n\chi(j\ze)^{-n-1}).
\eqa
The parameters of the simplified trapezoid rule are chosen as in Section \ref{s:sinh3}. 

To calculate $\tu(\pm z)$, we need to evaluate $d$ and $A_-(z)$ for $z=\pm \chi(j\ze)$, $ |j|\le N_\ell$.
The RHS in the formula \eq{def_d} for $d$ can be calculated using the trapezoid rule
but if the annulus of analyticity is narrow: $a-1<<1$, then it is advantageous to use \eq{ds} and the sinh-acceleration
\bbe\label{d_sinh}
d=-\frac{b_\ell}{2\pi}\int_{-\infty}^{+\infty}dy\, \cosh(i\om+y)\frac{\ln(A(\chi(y)A(\chi(-y))}{\chi(y)}.
\ee
 To evaluate $A_-(z)$, we can
use either \eq{def_Apm} and the trapezoid rule or \eq{WHFm}, sinh-changes of variables and the simplified trapezoid rule.
The simplest choice of the deformations in \eq{WHFm} is to use the same parameter sets
$(\sg,b,\om)=(\sg_\pm,b_\pm,\om_\pm)=(\sg_\ell,b_\ell,\om_\ell)$. In order that 
$\cL_{L;\sg_\ell, b_\ell,\om_\ell}\cap(1/\cL_{L;\sg_\ell, b_\ell,\om_\ell})=\emptyset$, we set $\om=-\ga/2, d_\ell=k_d\ga/2$ and choose $\sg_\ell, b_\ell$ so that 
$1\le \sg_\ell-b_\ell(\om_\ell+d_\ell)<a$. We set $r_-=1$, take $r_+\in (1,a)$, and use the prescription
\eq{bell_sinh_L}-\eq{sgell_sinh_L} to choose $b_\ell$ and $\sg_\ell$. Applying the simplified trapezoid rule to \eq{WHFm2} we can use the same $\ze$ as in \eq{hn_simp_trap}.
However, if $n-m$ is large, it is necessary to use  
longer grids to satisfy the same error tolerance. The reason is that 
$S_1(y):=\cosh(i\om+y)/(2\chi(y))\to 1/2$ as $y\to\pm\infty$, hence, bounded away from 0, and 
\[
S_2(y,z):=\ln A(\chi(y))/(z\chi(y)-1)-\ln A(-\chi(y))/(z\chi(y)+1)=O(|z|^{-1}e^{-|y|(\de'+1)})
\]
as $|z|+|y|\to \infty$, where $\de'\in (0,1)$ can be chosen close (but not very close) to 1. Hence, the truncation parameter must be chosen using $1+\de'$ instead of $n-m$, and the resulting $N^1$ is larger than $N$. 
The absolute error of the evaluation of $\ln A_-(z)$ translates into the relative error of $h[n]$,
therefore, if $h[n]$ is small, then in \eq{WHFm2}, a larger step and smaller number of terms can be used. 
Thus, we use
\bbe\label{Amhnsimple}
\ln A_-(z)\approx \frac{\ze^1}{2\pi}\sum_{-N^1\le j\le N^1}S_1(j\ze^1)S_2(j\ze^1,z).
\ee

\subsection{Algorithm}\label{ss:algo_filtering} Let $PSD$ satisfy Condition WHF-SINH3$(a;\ga;m)$
with $a>1$, $\ga\in (0,\pi/2)$, $m_\pm\in \bR$ and $c_\infty>0$. 
Given $\eps$ and a finite subset $\vec n:=[n_-, n_-+1, n_-+2,\ldots, n_+]$, where $n_->m:=m_++m_-$,
the array $\vec h=[h[n_-+j-1]]_{j=1}^{n_+-n_-+1}$ is calculated as follows.
\begin{enumerate}[I.]
\item
Choose parameters $\sg_\ell, b_\ell, \om_\ell$ of the sinh-deformation.
\item
Choose parameters $\ze, N$ and $\ze^1, N^1$ for the simplifying trapezoid rule in
the inverse $Z$-SINH formula for $h[n]$ and the Wiener-Hopf factor $A_-$. For the former case,
use $n=n_+$ to derive a bound for the Hardy norm and $n=n_-$ to find the truncation parameter.
\item
Set  $\vec y=\ze*(-N:1:N)$,
 $\chi=\sg+i*b*\sinh(i*\om+\vec y)$, $\vec{\mathrm{der}}=b*\cosh(i*\om+\vec y)$,
 $\vec y^1=\ze^1*(-N^1:1:N^1)$, 
$\chi^1=\sg^1+i*b^1*\sinh(i*\om^1+\vec y^1)$, 
 $\vec{\mathrm{der^1}}=b*\cosh(i*\om+\vec y^1)$.
 
\item
Set $\ka=a^{m_++m_-}/c_\infty$, $PSD^\pm=PSD(\pm \chi)$, $PSD^{\pm,1}=PSD(\pm \chi^1)$,
\beqast
A^{\pm}&=&
\ka*(a\mp \chi).^{-m_+}.*(a\mp 1./\chi).^{-m_+}
.*(a\pm \chi).^{-m_-}.*(a \pm 1./\chi).^{-m_-}.*PSD^{\pm},\\
A^{\pm,1}&=&\ka*(a\mp \chi^1).^{-m_+}.*(a\mp 1./\chi^1).^{-m_+}
.*(a\pm\chi^1).^{-m_-}.*(a\pm 1./\chi^1).^{-m_-}.*PSD^{\pm,1}.
\eqast
\item
Calculate $d$ applying the trapezoid rule to  \eq{def_d} or simplified trapezoid rule to \eq{d_sinh}
\[
d=-(\ze^1/2/\pi)*\mathrm{sum}(\log (A^+.*A^-).*\vec{\mathrm{der^1}}./\chi^1).
\]
\item
Calculate $A^\pm_-:=A_-(\pm \chi)$: 
\beqast
S0&=&\mathrm{conj}(\chi^1)'*\chi,\ FR^\pm=1./(S0\pm 1);\\
A^+_-&=&\exp[(\ze^1/(2*\pi))*((\vec{\mathrm{der^1}}.*\log (A^{+,1})*FR^--(\vec{\mathrm{der^1}}.*\log (A^{-,1}))*FR^+)],\\
A^-_-&=&\exp[(\ze^1/(2*\pi))*((\vec{\mathrm{der^1}}.*\log (A^{-,1})*FR^--(\vec{\mathrm{der^1}}.*\log (A^{+,1}))*FR^+)].
\eqast
\item
Calculate $\tu^\pm:=\tu(\pm \chi)$:
\[
\tu^\pm=e^{d/2}c_\infty^{-1/2}*a^{m/2}PSD^\pm./A^\pm_-.*(a\mp 1./\chi).^{-m_+}.*(a\pm 1./\chi).^{-m_-}.
\]
\item
Calculate $h=[h[\vec n]]$:
\beqast
Znp&=&\mathrm{conj}(\chi)'.^{-\vec n -1},\ Znm=\mathrm{conj}(-\chi)'.^{-\vec n -1};\\
h&=&(\ze/(2*\pi))*((\vec{\mathrm{der}}.*\tu^+)*Znp-(\vec{\mathrm{der}}.*\tu^-)*Znm). 
\eqast
\end{enumerate} 
\begin{example}\label{Ex:filter1}{\rm Let $PSD(z)=H(z)H(1/z)$, where $H(z)=(a_+-1/z)^{m_+}(a_-+1/z).^{m_-}$, and $a_+=1.0001, a_-=1.00015$, $m_+=3$, and $m_-=-1$. The impulse response function decays very slowly.
We evaluate $h[n]$ for $n\in [100,400]$ using the explicit formula for $H(z)$ and the trapezoid rule with $N=800,001$ terms.
The maximal relative error is 8.21E-15; if we take larger numbers of terms, the accuracy decreases due to the rounding errors.
With  $N=80,001$, the maximal relative error is 8.15E-06. 
The algorithm in this section with $N=172$ and $N^1=237$ achieves accuracy 4.55E-15; the CPU time for the evaluation of all $h[n]$, $n\in [100,400]$, is 13.4 msec., the average over 10,000 runs.

If we use the Wiener-Hopf factorization and trapezoid rule to recover $H(z)$ and then apply the trapezoid rule, then the accuracy of
the trapezoid rule significantly decreases.

}
\end{example}

\begin{example}\label{Ex:filter2}{\rm In Example \ref{Ex:filter1}, $H(z)$ has only one singularity outside the unit disc, and the performance of the trapezoid rule and method of the paper can be significantly improved making appropriate changes of the variable.
 If we let $m_\pm=-1$, then $H(z)$ has poles at $z=\pm 1$, and the analytic properties of the integrand are worse. In the result,
 the maximal relative error of the trapezoid rule with $N=80,001$ terms is only 0.067, and with $N=800,001$ terms - 8.01E-11.
 The method of this section, with the same parameters of the numerical scheme as in Example \ref{Ex:filter1}, achieves
 the relative accuracy 1.97E-11 in 13.1 msec.
 }
 \end{example}
 
 \begin{example}\label{Ex:filter3}{\rm In Example \ref{Ex:filter2}, we take $a_+=1.00001$ and $a_-=1.000015$; the impulse response function decays extremely slowly, and the annulus of analyticity is very narrow. The trapezoid rule
with $N=800,001$ that uses the explicit formula for $H(z)$, produces the maximal relative error 0.67. The method of this section with $N=575$ and $N^1=626$,  achieves
 the relative accuracy 4.08E-10 in 26.4 msec.
 
 }
 \end{example}

\section{Conclusion}\label{s:concl}
In the paper, we applied the conformal deformations technique and simplified trapezoid rule to construct efficient versions
of the inverse $Z$-transform and Wiener-Hopf factorization of functions on the unit circle.
With one exception (LOG-acceleration), we used the sinh-acceleration, which is the most efficient if
the integrand $f(z)$ in the inverse $Z$-tranform formula decays as $z\to \infty$  in a symmetric sector with a not very small opening angle. If the angle is very small, then, after the sinh-change of variables is made, the strip of analyticity
of the new integrand is too narrow, and the step in the infinite trapezoid rule must be very small to satisfy even moderately small error tolerance. Hence, the number of terms is very large. In these cases, a seemingly less efficient
family of fractional-parabolic deformations is more efficient. The changes of variables are of the form
$\chi^+_{P;\sg, b, \be, p}=
(\sg+b(1+iy)^\be)^p$, where $\sg>-b$,  and $\chi^-_{P;\sg, b, \be, p}=
(\sg-b(1-iy)^\be)^p$, where $\sg>b$. In some cases, the sub-polynomial deformations more general than
\eq{chi_log} are necessary. See \cite{ConfAccelerationStable} for application of different families
of conformal deformations to evaluation of stable L\'evy distributions. 
As applications, we considered evaluation of high moments of probability distributions and 
impulse response functions given the power spectral density. In the latter case, we had to design an efficient
numerical algorithm for the calculation of the Wiener-Hopf factors. Numerical examples demonstrate that
if the impulse response function decays slowly at infinity, equivalently, the annulus of analyticity of the power spectral density is very narrow (the shocks 
are persistent), then the trapezoid rule is inefficient. The method of the paper admits the straightforward
modification to the case when the open annulus of analyticity does not exist, and the impulse response function decays slower than exponentially. In the definitions and constructions of the paper,
it is possible to pass to the limit as the annulus shrinks into a circle. The sinh-deformed curves degenerate into angles, and on each ray, an exponential change of the variables is made,
similarly to \cite{ConfAccelerationStable,EfficientStableLevyExtremum}.
% Additional variations used in
%the paper are deformations and changes of variables defined by positive powers of the functions that define the sinh-%deformations and fractional-parabolic deformations. 
The new variations of the inverse $Z$-transform suggested in the paper can be also used in
situations where the first version of sinh-acceleration was applied. Namely, in \cite{EfficientDiscExtremum,EfficientZsufficient}, we considered 
the  evaluation of large powers of bounded operators, solution of boundary problems for difference-integro-differential equations and evaluation of lookback and barrier options with the discrete monitoring. In the latter case, it is necessary that the deformations of the lines of integration in the formulas for the Wiener-Hopf factors for random walks and in the formula for the inverse $Z$-transform be in a certain agreement. 
% The numerical schemes in the paper can be regarded as a further step in  the general program of study of the %efficiency
%of combinations of one-dimensional inverse transforms for high-dimensional inversions systematically pursued by
%Abate-Whitt, Abate-Valko \cite{AbWh,AbWh92OR,AbateValko04,AbWh06} and other authors.

%\bibliographystyle{siamplain}
%\bibliography{thebibliographyFM24}

\end{document}